\documentclass[12pt,oneside]{amsart}
\usepackage{geometry} 
\usepackage{subcaption}
\newcommand{\includegraphics}[2][]{\fbox{}}

\usepackage{amssymb,latexsym,amsmath,extarrows,amsthm, tikz}
\usepackage{pgfplots}
\usepackage{hyperref}
\usepackage{graphicx}
\usepackage{mathabx}
\usepackage{enumerate}
\usepackage{ulem} 
\usepackage{subcaption}

\DeclareMathAlphabet{\pazocal}{OMS}{zplm}{m}{n} 
\usepackage{color}
\definecolor{blue}{rgb}{0,0,1}

\definecolor{red}{rgb}{1,0,.2}

\theoremstyle{plain}
\newtheorem{thm}{Theorem}[section]
\newtheorem{lem}[thm]{Lemma}

\theoremstyle{definition}
\newtheorem{defn}[thm]{Definition}
\newtheorem{remark}[thm]{Remark}

\newtheorem{example}[thm]{Example}
\numberwithin{equation}{section}

\newcommand{\N}{\ensuremath{\mathbb{N}}}

\newcommand{\R}{\ensuremath{\mathbb{R}}}

\usepackage{hyperref}

\DeclareMathOperator{\K}{\mathcal{K}}

\newcommand{\eps}{\varepsilon}
\newcommand{\e}{\varepsilon}

\usepackage{hyperref}

\author[Alex McDonald]{Alex McDonald }
\address{Alex McDonald, Department of Mathematics, The Ohio State University}
\email{mcdonald.996@osu.edu}

\author[Krystal Taylor]{Krystal Taylor}
\address{Krystal Taylor, Department of Mathematics, The Ohio State University}
\email{taylor.2952@osu.edu}
\thanks{Taylor is supported in part by the Simons Foundation Grant 523555.}

\title{\parbox{14cm}{\centering{ 
Infinite Constant gap length trees in Products of Thick Cantor Sets 
}}}
\begin{document}

\maketitle
\begin{abstract}
We show that products of sufficiently thick Cantor sets generate trees in the plane with constant distance between adjacent vertices. 
Moreover, we prove that the set of choices for this distance has non-empty interior. 
We allow our trees to be countably infinite, which further distinguishes this work from previous results on patterns in fractal sets. 
This builds on the authors' previous work on graphs and distance sets over products of Cantor sets of sufficient Newhouse thickness.  
\end{abstract}  
\maketitle

\section{Introduction} 
The study of finite point configurations within sets satisfying certain size conditions is a classic subject matter.  
The basic problem is to determine minimal size conditions satisfied by a subset of Euclidean space that guarantee the existence or abundance of various finite point configurations within the set.  
For subsets of the integers, size is quantified using cardinality or density.
For instance, in \cite{Z06} Ziegler showed that one can recover every simplex similarity type and all sufficiently large scalings inside subsets of $\R^d$ of positive upper density. 
Her result builds on previous work by Furstenberg, Katznelson, and Weiss on configurations in sets of positive density \cite{FKW90}.  
In the continuous setting, it is a consequence of the Lebesgue density theorem that subsets of $\R^n$ of positive Lebesgue measure contain a translated and scaled copy of every finite point set for an interval worth of scalings \cite{Steinhaus20}.  
\\

In the fractal setting, 
the size of a set may be quantified by notions of dimension, such as Hausdorff dimension. 
A number of works have appeared guaranteeing the existence 
\cite{CLP16, IosLiu, IosMag}
or abundance 
\cite{GM22, GGIP, GIP, GIT19, GIT21, Areas}
of certain point configurations within sets satisfying sufficient dimensional assumptions.  
In \cite{BIT16}, Bennett, Iosevich, and the second listed author establish that a set of sufficiently large Hausdorff dimension contains arbitrarily long chains with vertices in the set.  In more detail, define
the set of $k$-chains of distances determined by a set $E$ by
\begin{equation}
\label{distancechain}
\{(|x^1-x^2|,\dots,|x^k-x^{k+1}|)\in \R^k:x^i\in E \text{ distinct}\}.
\end{equation}
It is established in \cite{BIT16} that if the Hausdorff dimension of $E$ is greater than $\frac{d+1}{2}$, then the
above set has non-empty interior in $\R^k$. 
We refer to such a result as an \textit{abundance result}, as it gives dimensional assumptions that guarantee the abundance of chains within a set. 
Moreover, the authors of \cite{BIT16} prove
if the Hausdorff dimension of a set $E\subset \R^d$ is greater than $\frac{d+1}{2}$, then for each $k\in \N$, there exists a non-empty open interval $I$ so that for each $t\in I$, 
$E$ contains a $k$-chain of \textit{constant} gap length $t$ (i.e., a chain satisfying 
$|x^i - x^{i+1}|=t$; Figure \ref{chain}). 
We refer to this as an \textit{existence result}.  One can analogously define a constant gap tree (Figure \ref{tree}).  In a subsequent work, Iosevich and the second listed author generalized the results of \cite{BIT16} to arbitrary tree configurations, see \cite{Trees}.  
\\
\begin{figure}[t!]
\begin{subfigure}[b]{0.45\linewidth}
\centering
\begin{tikzpicture}
\draw [fill] (0,0) circle [radius=0.1];
\draw [fill] (1,3) circle [radius=0.1];
\draw [fill] (3,.551) circle [radius=0.1];
\draw [fill] (4,3.551) circle [radius=0.1];

\draw[red] (0,0)--(1,3)--(3,.551)--(4,3.551);

\node[below] at (0,0) {$x^1$};
\node[above] at (1,3) {$x^2$};
\node[below] at (3,.551) {$x^3$};
\node[above] at (4,3.551) {$x^4$};

\end{tikzpicture}
\caption{Chain with constant gap length}
\label{chain}
\end{subfigure}
\begin{subfigure}[b]{0.45\linewidth}
\centering
\begin{tikzpicture}
\draw [fill] (3,3) circle [radius=0.1];

\draw [fill] (1,2) circle [radius=0.1];
\draw [fill] (3,.764) circle [radius=0.1];
\draw [fill] (5,2) circle [radius=0.1];

\draw [fill] (0,0) circle [radius=0.1];
\draw [fill] (2,-1.236) circle [radius=0.1];
\draw [fill] (3,-1.472) circle [radius=0.1];
\draw [fill] (4,-1.236) circle [radius=0.1];
\draw [fill] (6,0) circle [radius=0.1];
\draw [fill] (7,1) circle [radius=0.1];

\draw[red] (3,3)--(1,2);
\draw[red] (3,3)--(3,.764);
\draw[red] (3,3)--(5,2);
\draw[red] (1,2)--(0,0);
\draw[red] (3,.764)--(2,-1.236);
\draw[red] (3,.764)--(3,-1.472);
\draw[red] (3,.764)--(4,-1.236);
\draw[red] (5,2)--(6,0);
\draw[red] (5,2)--(7,1);

\node[above] at (3,3) {$x^1$};
\node[above left] at (1,2) {$x^2$};
\node[above left] at (3,.764) {$x^3$};
\node[above right] at (5,2) {$x^4$};
\node[below] at (0,0) {$x^5$};
\node[below] at (2,-1.236) {$x^6$};
\node[below] at (3,-1.472) {$x^7$};
\node[below] at (4,-1.236) {$x^8$};
\node[below] at (6,0) {$x^9$};
\node[below] at (7,1) {$x^{10}$};
\end{tikzpicture}
\caption{Tree with constant gap length}
\label{tree}
\end{subfigure}
\caption{Chains and trees}

\end{figure}
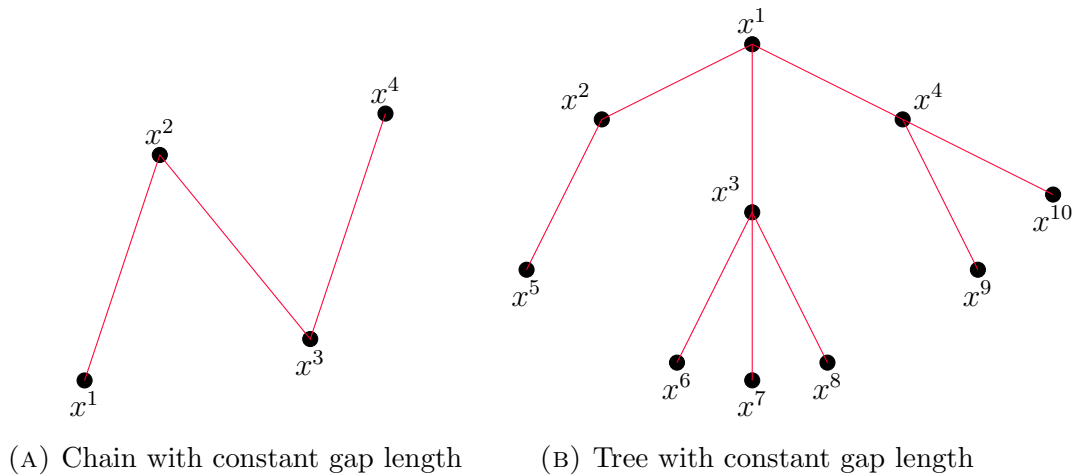

In this article, we study the existence and abundance of chains and trees of constant gap length under a different notion of size in the fractal setting known as \textit{Newhouse thickness}.
Assumptions on Newhouse thickness have yielded several interesting results on patterns contained in Cantor sets. 
 The precise definition of the thickness $\tau(K)$ of a Cantor set $K\subset \R$, as well as the precise definition of a Cantor set, will be given in the next section, but we first describe some of this work.
Simon and the second listed author \cite{ST20} proved the following foundational result.
\begin{thm}[Simon-Taylor, \cite{ST20}]
\label{STthm}
Let $K_1,K_2\subset\R$ be Cantor sets satisfying $\tau(K_1)\tau(K_2)>1$.  For any $x\in\R^2$, the pinned distance set
\[
\Delta_x(K_1\times K_2):=\{|x-y|:y\in\K_1\times K_2\}
\]
has non-empty interior.
\end{thm} 
In our previous paper \cite{MT21} we generalize Theorem \ref{STthm} from single distances to finite trees, showing that if $K_1, K_2$ satisfy the thickness condition in that theorem then the set of ``distance trees'' generated by $K_1\times K_2$ has non-empty interior (see the next section for precise definitions).  
In another direction, Yavicoli \cite{Y21} proved that for any finite point configuration in $\R$, certain sets above a thickness threshold contain a similar copy of that configuration (the threshold depends on the configuration, and is generally very large). 
In \cite{Yav22}, Yavicoli gives a definition of thickness in $\R^d$ that works even for totally disconnected sets satisfying a uniform denseness assumption, and proves there that sufficiently thick sets generate an interval worth of distances in any direction. 
Thickness theorems have applications in dynamical systems and fractal geometry \cite{Takahashi17}, as well as number theory \cite{Jiang22, Yu20}.
\\
%

The results of this article are a continuation of our work in \cite{MT21}.  In that work we showed that the hypothesis of Theorem \ref{STthm} was enough to ensure an abundance of chains and trees, and we posed the question of whether one could also obtain chains and trees with constant gap length.  In this article we partially answer this question in the affirmative.

\subsection{Acknowledgements}

We are grateful to the referee for all of their helpful comments and efforts to improve this manuscript.

\section{Definitions \& Main results}

Let $G=(V,\mathcal{E})$ be a graph, and let $\sim$ denote the adjacency relation (i.e., $i\sim j$ if and only if $i,j\in V$ and $(i,j)\in  \mathcal{E}$.  Given a set $E\subset\R^d$, define
\[
\Delta_G(E):=\{(|x^i-x^j|)_{i\sim j}: x^i \in E, x^i\neq x^j \text{ if } i\neq j\},
\]
where $\sim$ denotes the adjacency relation on the graph $G$, and $(|x^i-x^j|)_{i\sim j}$ is the vector with coordinates $|x^i-x^j|$ indexed by the edges of $G$.  In \cite{MT21}, the authors prove that if $G=T$ is a tree with $k+1$ vertices and $k$ edges, and $E=K_1\times K_2$ is a product of Cantor sets $K_1,K_2\subset \R$ with thickness (Definition \ref{thicknessdfn} below) satisfying $\tau(K_1)\cdot \tau(K_2)>1$, then $\Delta_T(K_1\times K_2)$ has non-empty interior.  In this paper, we consider a variant of this problem.  Given $E\subset\R^d$ and $t\in \R$, let $G_t(E)$ denote the graph with vertex set $E$ and an edge connecting $x,y\in E$ if and only if $|x-y|=t$.  Given a finite or infinite graph $G$, define
\[
\Delta_G'(E)=\{t\in\R: G_t(E) \text{ contains a subgraph isomorphic to } G\}.
\]
Note that two graphs are isomorphic if there is a bijection between their vertex sets which preserves adjacency.  This new distance set corresponds to the diagonal elements of the previous distance set:
\[
\Delta_G'(E)=\{t\in\R:(t,\dots,t)\in\Delta_G(E)\}.
\]
Accordingly, we refer to it as the \textit{diagonal} distance set. 

\begin{defn}
A \textbf{tree} is a connected acyclic graph.  
Equivalently, $T$ is a tree if and only if any two distinct vertices are connected by exactly one path. 
\end{defn}

\begin{defn}
\label{thicknessdfn}
A \textbf{Cantor set} is a non-empty subset of $\R^d$ which is compact, perfect, and totally disconnected.  A \textbf{gap} of a Cantor set $K\subset \R$ is a connected component of the complement $\R\setminus K$.  If $u$ is the right endpoint of a bounded gap $G$, for $b\in \R\cup \{\infty\}$, let $(a,b)$ be the closest gap to $G$ with the property that $u<a$ and $|G|\le b-a$ (Figure \ref{thicknessfig}).  The interval $[u,a]$ is called the \textbf{bridge} at $u$ and is denoted $B(u)$.  Analogous definitions are made when $u$ is a left endpoint.  The \textbf{thickness} of $K$ at $u$ is the quantity
\[
\tau(K,u):=\frac{|B(u)|}{|G|}.
\]
Finally, the thickness of the Cantor set $K$ is the quantity
\[
\tau(K):=\inf_u \tau(K,u),
\]
the infimum being taken over all endpoints $u$ of bounded gaps.
\end{defn}
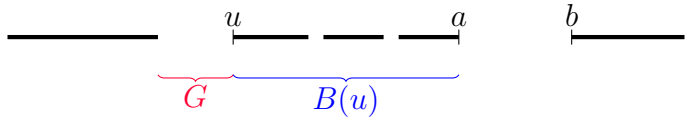
\begin{figure}[h]
\centering
\begin{tikzpicture}
\draw[ultra thick] (0,0)--(2,0);
\draw[ultra thick] (3,0)--(4,0);
\draw[ultra thick] (4.2,0)--(5,0);
\draw[ultra thick] (5.2,0)--(6,0);
\draw[ultra thick] (7.5,0)--(9,0);

\draw (3,-.1)--(3,.1);
\draw (6,-.1)--(6,.1);
\draw (7.5,-.1)--(7.5,.1);

\node[above] at (3,0) {$u$};
\node[above] at (6,0) {$a$};
\node[above] at (7.5,0) {$b$};

\draw [decorate,
	decoration = {brace}, red] (3,-.5) --  (2,-.5);
\draw [decorate,
	decoration = {brace}, blue] (6,-.5) --  (3,-.5);
	
\node[below, red] at (2.5,-.5) {$G$};
\node[below, blue] at (4.5, -.5) {$B(u)$};
\end{tikzpicture}
\caption{A gap and corresponding bridge}
\label{thicknessfig}
\end{figure}

It is worth noting that large thickness implies large Hausdorff dimension but not conversely.  
Specifically, one can prove the bound (see \cite[pg. 77]{PTbook}):
\[
\dim_{\rm H}(K) \geq \frac{\log{2}}{\log{\left( 2 + \frac{1}{\tau(K)} \right) }}.
\]

%
%
In \cite{MT21}, we studied the distance set $\Delta_T(K_1\times K_2)$ using the Newhouse gap lemma; this lemma states that if sets $K_1,K_2$ satisfy the condition $\tau(K_1)\tau(K_2)>1$, then either one set is contained in the other or their intersection is non-empty. 
In order to study the diagonal distance set $\Delta_T'(K_1\times K_2)$, however, we will need to ensure the intersection $K_1\cap K_2$ contains many points of a certain type.  While the Newhouse gap lemma gives non-empty intersection, it is possible for the intersection to be a single point (for example, see \cite{William}), and the gap lemma does not give us much control over what that point is.  Therefore, we will use the following result of Hunt, Kan, and Yorke \cite{HKY93}, which gives a thickness condition which ensures robust intersection.  We start by defining this thickness condition.
%
\begin{defn}
\label{HKYdfn}
The pair $(\tau_1,\tau_2)\in\R_{>0}^2$ is said to satisfy the \textbf{Hunt-Kan-Yorke thickness condition} (briefly, the HKY condition) if the following inequalities hold:
\begin{align*}
\tau_1&\geq \tau_2 \\
\tau_1&> \frac{\tau_2^2+3\tau_2+1}{\tau_2^2} \\
\tau_2&> \frac{(2\tau_1+1)^2}{\tau_1^3}.
\end{align*}
We say Cantor sets $K_1$ and $K_2$ satisfy the Hunt-Kan-Yorke thickness condition if either $(\tau(K_1),\tau(K_2))$ or $(\tau(K_2),\tau(K_1))$ satisfy the condition.
\end{defn}
\begin{remark}
The following observations are easy to verify:
\begin{itemize}
\item The Hunt-Kan-Yorke condition is monotone: if $(\tau_1,\tau_2)$ satisfies the condition and $\tau_1'\geq\tau_1,\tau_2'\geq\tau_2, \tau_1'\geq\tau_2'$, then $(\tau_1',\tau_2')$ satisfies the condtion as well.
\item The Hunt-Kan-Yorke condition implies the Newhouse condition $\tau_1\cdot \tau_2>1$ (in fact, it implies $\tau_1\cdot\tau_2> 5$), but not conversely.  More generally, there is no constant $c$ such that $\tau_1\cdot \tau_2>c$ implies the Hunt-Kan-Yorke condition.
\item If $\tau_1\geq \tau_2>\sqrt{2}+1$, then $(\tau_1,\tau_2)$ satisfies the Hunt-Kan-Yorke condition.
\item For any $\tau_2>0$, there exists $\tau_1>0$ such that $(\tau_1,\tau_2)$ satisfies the Hunt-Kan-Yorke condition.  As $\tau_2\to 0$, we must have $\tau_1\gtrsim 1/\tau_2^2$, so this is asymptotically worse then the Newhouse thickness condition $\tau_1>1/\tau_2$.
\end{itemize}
\end{remark}
This stronger thickness condition guarantees a more robust intersection of the sets $K_1$ and $K_2$ than is given by the Newhouse gap lemma.

\begin{defn}
Cantor sets $K_1,K_2\subset \R$ are said to be \textbf{interleaved} each has non-empty intersection with the interior of the convex hull of the other.
\end{defn}

\begin{lem}[Hunt-Kan-Yorke intersection theorem]
\label{HKY}
Let $K_1,K_2\subset\R$ be interleaved Cantor sets satisfying the Hunt-Kan-Yorke thickness condition.  Then, there exists a Cantor set $K\subset K_1\cap K_2$ satisfying $\tau(K)>0$.
\end{lem}
Our main theorem is as follows.
\begin{thm}
\label{MT2}
Let $T$ be a countable tree, and let $K_1,K_2\subset\R$ be Cantor sets satisfying the Hunt-Kan-Yorke thickness condition.  For $j=1,2$, if $\widetilde{K}_j\subset K_j$ is a Cantor set which does not contain the minimum or maximum element of $K_j$ and satisfies $\tau(\widetilde{K}_j)\geq \tau(K_j)$, then 
\[
\Delta(\widetilde{K}_1\times \widetilde{K}_2)\subset \Delta_T'(K_1\times K_2).
\]
In particular, by Theorem \ref{STthm}, the set $\Delta_T'(K_1\times K_2)$ has non-empty interior.
\end{thm}

Theorem \ref{MT2} says that diagonal distance set of $K_1\times K_2$ will have non-empty interior whenever the original single-distance set of an appropriate subset has non-empty interior; Theorem \ref{STthm} simply gives us one way to recognize when this occurs.  
Note that Lemma \ref{bubble} guarantees the existence of $\widetilde{K}_1, \widetilde{K}_2$ satisfying the hypotheses of Theorem \ref{MT2}. In general, if $I$ is a bridge in the construction of a Cantor set, $K$, then $K\cap I$ is Cantor and $\tau(K) \le \tau(K\cap I)$.  

\begin{example}
Let $C_\alpha$ denote the middle $\alpha$-Cantor set obtained by removing the middle $\alpha$ portion of the unit interval $[0,1]$ and iterating; 
it is not hard to check that the thickness of such a set is $\frac{1 - \alpha}{2\alpha}$. 
If $\tau(C_\alpha)  > 1 + \sqrt{2}$, then the Hunt-Kan-Yorke (HKY) condition is satisfied by $C_\alpha$ paired with itself, which, rounding, holds when  
$\alpha \le 1/6$. So, $C_\frac{1}{6} \times C_\frac{1}{6}$ contains arbitrarily long trees with an interval worth of admissible gap lengths. 
Note that $C_{1/6}$ paired with itself satisfies both the Newhouse and HKY condition, while $C_{1/4}$ only satisfies the Newhouse condition. 
\end{example}


\begin{remark}
Our method of proof relies on the intersection of two Cantor sets containing many points, which is guaranteed by the HKY condition.  
It is not clear if the assumption $\tau(K_1)\cdot \tau(K_2)>1$ is sufficient for any tree in light of the fact that, in general, such sets could intersect in a single point (see \cite{William}). 
\end{remark}

\section{Proofs}
Our proof strategy is similar to the one used in \cite{ST20} and \cite{MT21}.  The idea is to translate the statement that a distance is achieved into the statement that sets $K_2$ and $g(K_1)$ intersect for an appropriate function $g$.  In order to get an intersection point, we first need to ensure that the relevant thickness condition is satisfied, which means we need to understand how thickness changes under mapping.  It is easy to see that if $G$ is a gap of $K_1$, then $g(G)$ is a gap of $g(K_1)$.  However, if $B$ is the corresponding bridge, it does not follow that $g(B)$ is also a bridge.  This is because $g$ can distort lengths, so there may be a gap $G'\subset B$ which satisfies $|G'|<|G|$ but $|g(G')|>|g(G)|$.  However, the key observation is that if we work locally, small intervals which are close together should be distorted by a similar amount.  Therefore, thickness is ``locally almost preserved''.  The following lemma, a variant of \cite[Lemma 3.8]{ST20} and \cite[Lemma 3.4]{MT21}, makes this idea precise.
\begin{lem}[Thickness is locally almost preserved by smooth maps]
\label{image}
Let $K$ be a Cantor set, and let $I$ be an interval such that $\widetilde{K}:=K\cap I$ is a Cantor set with thickness $\tau(\widetilde{K})\geq \tau(K)$.  For any $0<c<1$, there exists $\e>0$ such that if $g$ is a continuously differentiable monotone function on $I$ which satisfies
\[
\left|\frac{|g'(x)|}{|g'(y)|}-1\right|<\e
\]
for all $x,y\in I$, then
\[
\tau(g(\widetilde{K}))>c\tau(K).
\]
\end{lem}
\begin{proof}[Proof of Lemma \ref{image}]
Fix $0<c<1$, and let $I$ and $g$ satisfy the hypotheses of the theorem for some $\e>0$ to be determined later.  We must show that $\tau(g(\widetilde{K}))>c\cdot \tau(K)$ if $\e$ is sufficiently small.  The main idea behind the proof is the notion of ``$\e$-thickness'' defined in \cite[Definition 3.6]{ST20}.  Given a gap $G$ of a Cantor set $K$ with right endpoint $u$, let $(a,b)$ be the closest gap to $G$ which satisfies $a>u$ and 
$b-a>(1-\e)|G|$.
  The interval $[u,a]$ is called the $\e$-bridge of $K$ at $u$, denoted $B_\e(u)$.  The $\e$-bridge at a left endpoint is defined analogously.  The $\e$-thickness of $K$ at $u$ is 
\[
\tau_\e(K,u):=\frac{|B_\e(u)|}{|G|},
\]
and the $\e$-thickness of $K$ is $\tau_e(K):=\inf_u \tau_\e(K,u)$.  Thus, the thickness $\tau(K)$ is simply the $0$-thickness, and $\tau_\e(K)\nearrow \tau(K)$ as $\e\searrow 0$.  By the mean value theorem, for any subinterval $J\subset I$ there exists $x_J\in J$ with $|g(J)|=|g'(x_J)|\cdot |J|$.  Let $G$ be a bounded gap of $\widetilde{K}$ with endpoint $u$.  If $H\subset B_\e(u)$ is another gap, the by definition we have $|H|\leq(1-\e)|G|$. 
 It follows that
\begin{align*}
|g(H)|&=|H|\cdot |g'(x_H)| \\
&\leq (1-\e)|G|\cdot \frac{|g'(x_H)|}{|g'(x_G)|}\cdot |g'(x_G)| \\
&<(1-\e)|G|\cdot (1+\e)\cdot |g'(x_G)| \\
&=(1-\e^2)|g(G)|.
\end{align*}
Since any gap in $g(B_\e(u))$ is of the form $g(H)$ for some gap $H$ in $B_\e(u)$, this calculation shows that $g(B_\e(u))$ is contained in the $\e^2$-bridge at $g(u)$: 
\[
g(B_\e(u)) \subset B_{\eps^2}(g(u)).
\]
Therefore, the $\e^2$-thickness of $g(\widetilde{K})$ at $g(u)$ satisfies
\begin{align*}
\tau_{\e^2}(g(\widetilde{K}),g(u))&=\frac{|B_{\e^2}(g(u))|}{|g(G)|} \\
&\geq \frac{|g(B_\e(u))|}{|g(G)|} \\
&=\frac{|B_\e(u)|}{|G|}\cdot\frac{|g'(x_{B_\e(u)})|}{|g'(x_G)|} \\
&>\tau_\e(\widetilde{K},u)(1-\e).
\end{align*}
Taking infima over $u$, we get $\tau_{\e^2}(g(\widetilde{K}))\geq \tau_\e(\widetilde{K})(1-\e)$.  Since $\tau(g(\widetilde{K}))\geq \tau_{\e^2}(g(\widetilde{K}))$ for any $\e>0$ (this is not specific to $\widetilde{K}$, thickness is always greater than $\e$-thickness by definition), it follows that 
\[
\tau(g(\widetilde{K}))\geq \tau_\e(\widetilde{K})(1-\e)\geq \tau_\e(K)(1-\e) .
\]
Since $\tau_\e(K)(1-\e)\to \tau(K)$ as $\e\to 0$, we have $\tau_\e(K)(1-\e)>c\tau(K)$ if $\e$ is sufficiently small.
\end{proof}
The key to applying the lemma in our context is, given sets $K_1$ and $K_2$ and function $g$, to choose a small interval $I$ and consider $\widetilde{K}_1:=K_1\cap I$, where $I$ is as in the hypotheses of Lemma \ref{image}.  Since the Hunt-Kan-Yorke condition is monotone, the new sets will satisfy the condition also.  Finally, since the Hunt-Kan-Yorke condition is an ``open condition,'' $g(\widetilde{K}_1)$ and $\widetilde{K}_2$ will satisfy it as well, provided $c$ is sufficiently close to $1$.  Unfortunately, it is not true in general $\tau(K\cap I)\geq \tau(K)$ for an arbitrary interval $I$.  
For example, the middle-third Cantor set has thickness equal to $1$, but its intersection with $[\frac{2}{9}, \frac{7}{9}]$  has thickness strictly less than $1$. 
However, this property does always hold if $I$ is replaced with a carefully chosen subinterval $J$.  More precisely, we have the following lemma.
\begin{lem}
\label{bubble}
Let $K\subset \R$ be a Cantor set, let $x\in K$, and let $I\subset \R$ be an open interval containing $x$.  There exists a compact interval $J\subset I$, also containing $x$, such that $K':=K\cap J$ is a Cantor set satisfying $\tau(K')\geq \tau(K)$.  If $x$ is a gap endpoint of $K$, then $J$ may be chosen so that $x$ is a gap endpoint of $K'$.  If $x$ is not a gap endpoint of $K$, then $J$ may be chosen so that $x$ is not a gap endpoint of $K'$.
\end{lem}

\begin{proof}
First assume $x$ is an endpoint of a gap $G$; without loss of generality suppose $x$ is the right endpoint of $G$.  Since gaps are disjoint, the interval $I$ can intersect only finitely many gaps to the rigth of $x$ of length at least $|G|$.  Let $G'$ be the leftmost such gap (if all bounded gaps to the right of $x$ are smaller than $G$, let $G'$ be a gap intersecting $I$ of maximal length).  Let $y$ be the left endpoint of $G'$, let $J=[x,y]$, and let $K'=J\cap K$.  Clearly $K'$ is a Cantor set, so it remains to show $\tau(K')\geq \tau(K)$.  Since all bounded gaps of $K'$ are smaller than $|G|$ and $|G'|$ the bridge at any gap endpoint $u$ with respect to $K'$ is the same as the bridge at $u$ with respect to $K$.  It follows that $\tau(K,u)=\tau(K',u)$ for each bounded gap endpoint $u$ of $K'$.  Taking infima over bounded gap endpoints of $K'$ on the right and the larger class of bounded gap endpoints of $K$ on the left gives $\tau(K)\leq \tau(K')$. \\

Now, suppose $x$ is not a gap endpoint of $K$.  Let $G_L$ be a gap to the left of $x$ which is contained in $I$, and let $G_R$ be a gap to the right of $x$ which is contained in $I$.  Let $\ell=\min(|G_L|,|G_R|)$.  There are at most finitely many gaps between $G_L$ and $x$ of length at least $\ell$; let $G_L'$ be the rightmost such gap, and let $u$ be its right endpoint (if there are no such gaps, take $G_L'=G_L$.  Define $G_R'$ similarly, and let $v$ be its left endpoint.  Let $J=[u,v]$, and let $K'=K\cap J$.  As in the previous case, since each bounded gap of $K'$ is smaller than $|G_L'|$ and $|G_R'|$, the bridges with respect to $K'$ coincide with the bridges with respect to $K$, and therefore $\tau(K)\leq \tau(K')$.
\end{proof}
We are now ready to begin constructing trees of constant gap length.  Our first lemma requires a little more terminology.
\begin{defn}
Let $K_1,K_2\subset \R$ be Cantor sets, and assume $\tau(K_1)\geq \tau(K_2)$.  A point $x=(x_1,x_2)\in K_1\times K_2$ is called a \textbf{special point} if $x_1$ is not a gap endpoint of $K_1$, and $x_2$ is not the maximum or minimum point of $K_2$.
\end{defn}
Note that if $x_1 \in K_1$ is \textit{not} an endpoint of a gap, then it is both a left and right-limit point of $K_i$.  Moreover, it is also a left or right limit of left or right gap endpoints of $K_i$.  This ensures that when we define $\widetilde{K}_1:=K_1\cap I$ for some interval $I$, we can assume $x_1$ is not the maximum or minimum point of $\widetilde{K}_1$.  The following lemma states that, given arbitrary points $x, y \in K_1\times K_2$, provided $x$ is special, then the circle about $y$ through $x$ contains infinitely many points of $K_1 \times K_2$.  This allows us to generate arbitrarily large ``star-shaped graphs''.

\begin{lem}[Star-shaped graph lemma]
\label{star}
Let $K_1,K_2$ be Cantor sets satisfying the Hunt-Kan-Yorke condition.  Let $y\in\R^2$, let $x\in K_1\times K_2$ be a special point which does not share a coordinate with $y$, and let $t=|x-y|$.  Then, there are uncountably many $z\in K_1\times K_2$ such that $|y-z|=t$.  
In particular, there are uncountably many special points, $z$, such that $|z-y|=t$
\end{lem}
\begin{figure}[h]
\centering
\begin{tikzpicture}
\draw[<->] (0,5)--(0,0)--(5,0);

\draw [fill] (2.967,3.657) circle [radius=0.1];
\draw [fill] (1,1.625) circle [radius=0.1];

\draw [blue] (2,1) arc [radius=2.828, start angle=250, end angle= 180];

\draw [thick] (1,-.1)--(1,.1);
\draw [thick] (-.1,1.625)--(.1,1.625);

\node[below] at (1,0) {$x_1$};
\node[left] at (0,1.625) {$x_2$};
\node[above right] at (1,1.625) {$x$};
\node[right] at (3.067,3.657) {$y$};
\node[below, red] at (.5,0) {$u_1$};
\node[left, red] at (0,2.274) {$g_{t,y}(u_1)$};
\node[below, red] at (1.5,0) {$v_1$};
\node[left, red] at (0,1.239) {$g_{t,y}(v_1)$};

\draw[dashed, red] (.5,0)--(.5,2.274);
\draw[dashed, red] (0,2.274)--(.5,2.274);
\draw[dashed, red] (1.5,0)--(1.5,1.239);
\draw[dashed, red] (0,1.239)--(1.5,1.239);

\end{tikzpicture}
\caption{Constructing the sets $\widetilde{K}_j$.}
\label{interleave}
\end{figure}
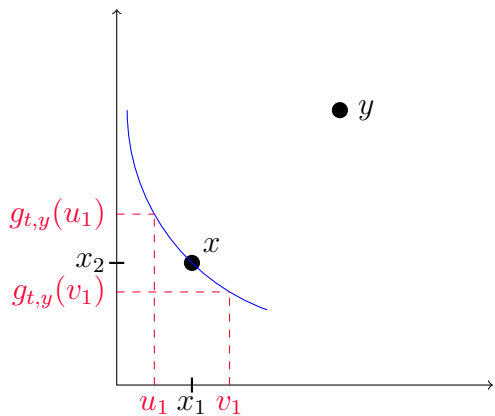

\begin{proof}
Assume $x_1<y_1$ and $x_2<y_2$, so the point $x=(x_1,x_2)$ is below and to the left of $y=(y_1,y_2)$.  Because thickness does not change under symmetries, there is no loss of generality in this assumption.  Define a function $g_{y,t}:(y_1-t,y_1+t)\to\R$ by
\[
g_{y,t}(s)=y_2-\sqrt{t^2-(s-y_1)^2},
\]
so that the graph of $g_{y,t}$ is the lower half circle centered at $y$ of radius $t$.  In particular, $g_{y,t}(x_1)=x_2$.  Note that $g_{y,t}$ is smooth, and the derivative $g_{y,t}'$ vanishes only at the point $y_1$.  By Lemma \ref{bubble}, we may choose an interval $J$ containing $x_1$ in its interior such that $\widetilde{K}_1:=K_1\cap J$ and $K_2$ satisfy the Hunt-Kan-Yorke thickness condition, and $x_1$ is not a gap endpoint of $\widetilde{K}_1$.  We may also assume $J$ is as small as needed.  Since $x_1$ is not a critical point of $g_{y,t}$, we may assume that $J$ satisfies the conditions of Lemma \ref{image}, and hence $g_{y,t}(\widetilde{K}_1)$ and $K_2$ satisfy the Hunt-Kan-Yorke condition as well.  Since $|y-z|=t$ whenever $g_{y,t}(z_1)=z_2$, the statement of the lemma follows from the claim that $g_{y,t}(\widetilde{K}_1)\cap K_2$ is uncountable.  To prove this claim, suppose this set is countable.  We will show $g_{y,t}(\widetilde{K}_1)$ and $K_2$ are interleaved, contradicting the Hunt-Kan-Yorke intersection theorem.  Since $x_1$ is not a gap endpoint of $K_1$, it is not the maximum or minimum of $\widetilde{K}_1$, hence the set $\widetilde{K}_1$ contains uncountably many points to both the left and right of $x_1$.  If $g_{y,t}(\widetilde{K}_1)$ intersects $K_2$ in only countably many places, then there exist $u_1,v_1\in \widetilde{K}_1$ such that $u_1<x_1<v_1$ and $g_{y,t}(u_1),g_{y,t}(v_1)\notin K_2$ (Figure \ref{interleave}).  Since $g_{y,t}$ is monotone decreasing in a neighborhood of $x_1$, we have $g_{y,t}(v_1)<x_2<g_{y,t}(u_1)$.  Since $x_2\in K_2$, each of the points $g_{y,t}(u_1)$ and $g_{y,t}(v_1)$ are contained in a different gap of $K_2$.  A similar argument shows that there are points $u_2,v_2\in K_2$ contained in different gaps of $g_{y,t}(\widetilde{K}_1)$.  The Hunt-Kan-Yorke intersection theorem (Lemma \ref{HKY}) then guarantees the intersection $g_{y,t}(\widetilde{K}_1)\cap K_2$ contains a Cantor set of positive thickness, contradicting countability. To prove the ``in particular" part of the lemma, observe that any graph of a function can contain at most countably many points of $K_1\times K_2$ that are not special points.

\end{proof}

To prove our main theorem, we will start with a single distance between two points and use it to build a tree one star-shaped graph at a time.  This will almost prove the theorem, which says that ``almost'' any distance is an admissable gap length for a tree.  However, our Lemma \ref{star} only allows us to use special points, so we have to ensure that we have not lost too many distances.  This is accomplished by the following lemma.

\begin{lem}
\label{special}
Let $K_1,K_2\subset \R$ be Cantor sets satsifying the Hunt-Kan-Yorke condition and assume $\tau(K_1)\geq \tau(K_2)$.  For $j=1,2$, let $\widetilde{K}_j\subset K_j$ be a Cantor set which does not contain the maximum or minimum of $K_j$ and satisfies $\tau(\widetilde{K}_j)\geq \tau(K_j)$, and let $\mathcal{S}$ be the set of special points of $K_1\times K_2$.  Then, $\Delta(\widetilde{K}_1\times\widetilde{K}_2)\subset \Delta(\mathcal{S})$.
\begin{proof}
Let $t\in \Delta(\widetilde{K}_1\times\widetilde{K}_2)$, say $t=|x-y|$ with $x,y\in \widetilde{K}_1\times\widetilde{K}_2$.  Let $r=x_1-y_1$ and $s=x_2-y_2$, so that $t^2=r^2+s^2$.  By assumption, $x_2$ and $y_2$ are not the maximum or minimum points of $K_2$, so to prove the lemma it suffices to prove there exist $x_1',x_2'\in K_1$ which are not gap endpoints and satisfy $x_1'-y_1'=r$, since the points $(x_1',x_2)$ and $(y_1',y_2)$ would be special points with distance $t$.  By monotonicity of the Hunt-Kan-Yorke condition and the assumption $\tau(K_1)\geq \tau(K_2)$, the set $K_1$ satisfies the Hunt-Kan-Yorke condition when paired with itself.  Finally, since $\widetilde{K}_1$ does not contain the maximum or minimum points of $K_1$, we have $|r|<(\max K_1-\min K_1)$.  Therefore, the Hunt-Kan-Yorke intersection theorem implies that $K_1$ and $r+K_1$ have uncountably many intersection points, meaning there are uncountably many pairs $(x_1',y_1')\in K_1\times K_1$ which satisfy the equation $x_1'=y_1'+r$.  Since each choice of $y_1'$ yields a different $x_1'$ and there are only countably many gap endpoints, we conclude that we may find a solution such that neither $x_1'$ nor $y_1'$ is a gap endpoint.
\end{proof}
\end{lem}


\subsection{Proof of Theorem \ref{MT2}}

\begin{proof}[Proof of Theorem \ref{MT2}]
\begin{figure}[h!]
\begin{tikzpicture}
\draw[fill] (8.5,5) circle[radius=0.1];
\draw[fill] (2.5,2.5) circle[radius=0.1];
\draw[fill] (6.5,2.5) circle[radius=0.1];
\draw[fill] (10.5,2.5) circle[radius=0.1];
\draw[fill] (1,0) circle[radius=0.1];
\draw[fill] (2,0) circle[radius=0.1];
\draw[fill] (3,0) circle[radius=0.1];
\draw[fill] (5,0) circle[radius=0.1];
\draw[fill] (6,0) circle[radius=0.1];
\draw[fill] (7,0) circle[radius=0.1];
\draw[fill] (9,0) circle[radius=0.1];
\draw[fill] (10,0) circle[radius=0.1];
\draw[fill] (11,0) circle[radius=0.1];
\draw (8.5,5)--(2.5,2.5);
\draw (8.5,5)--(6.5,2.5);
\draw (8.5,5)--(10.5,2.5);

\draw (2.5,2.5)--(1,0);
\draw (2.5,2.5)--(2,0);
\draw (2.5,2.5)--(3,0);

\draw (6.5,2.5)--(5,0);
\draw (6.5,2.5)--(6,0);
\draw (6.5,2.5)--(7,0);

\draw (10.5,2.5)--(9,0);
\draw (10.5,2.5)--(10,0);
\draw (10.5,2.5)--(11,0);
\node[above right] at (8.5,5) {$\varnothing$};
\node[left] at (2.5,2.5) {$1$};
\node[left] at (6.5,2.5) {$2$};
\node[right] at (10.5,2.5) {$3$};
\node[below] at (1,0) {$(1,1)$};
\node[below] at (2,0) {$(1,2)$};
\node[below] at (3,0) {$(1,3)$};
\node[below] at (5,0) {$(2,1)$};
\node[below] at (6,0) {$(2,2)$};
\node[below] at (7,0) {$(2,3)$};
\node[below] at (9,0) {$(3,1)$};
\node[below] at (10,0) {$(3,2)$};
\node[below] at (11,0) {$(3,3)$};
\node at (1,-1) {$\vdots$};
\node at (2,-1) {$\vdots$};
\node at (3,-1) {$\vdots$};
\node at (5,-1) {$\vdots$};
\node at (6,-1) {$\vdots$};
\node at (7,-1) {$\vdots$};
\node at (9,-1) {$\vdots$};
\node at (10,-1) {$\vdots$};
\node at (11,-1) {$\vdots$};
\node at (3.5,0) {$\cdots$};
\node at (7.5,0) {$\cdots$};
\node at (11.5,0) {$\cdots$};
\node at (11.5,2.5) {$\cdots$};
\end{tikzpicture}
\caption{The tree $T^*$}
\label{Tstar}
\end{figure}
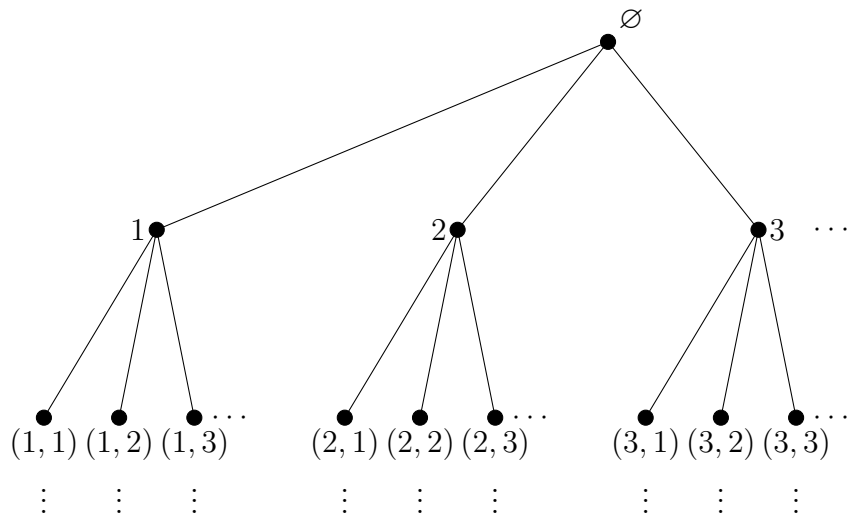 
Consider the tree $T^*$ with vertex set consisting of all finite sequences of positive integers (including the empty sequence, denoted $\varnothing$), with an edge between two sequences if and only if they are of the form $(i_1,\dots, i_n)$ and $(i_1,\dots,i_n,i_{n+1})$ (or $\varnothing$ and $(i_1)$) for $i_1,\dots,i_{n+1}\in\N$ (Figure \ref{Tstar}).  It is easy to see that every countable tree is a subgraph of $T^*$, so it suffices to prove the theorem in the case $T=T^*$.  To prove this case, for any non-negative integer $n$, let $L_n$ denote the $n$-th ``level'' of $T^*$, i.e., the set of sequences of lengh exactly $n$ (the empty sequence has length $0$).  For each $w\in L_n$, let $L_{n+1}(w)$ be the sequences in $L_{n+1}$ which are adjacent to $w$; explicitly, if $w=(i_1,\dots,i_n)$, then
\[
L_{n+1}(w)=\{(i_1,\dots,i_n,i_{n+1}):i_{n+1}\in\N\}.
\]
By lemma \ref{special}, to prove the theorem it suffices to show $\Delta(\mathcal{S})\subset \Delta_{T^*}'(K_1\times K_2)$, where $\mathcal{S}$ is the set of special points of $K_1\times K_2$.  Given $t\in \Delta(\mathcal{S})$, for each finite sequence $w$ we construct a point $x^w\in K_1\times K_2$ such that $|x^w-x^{w'}|=t$ whenever $w\sim w'$.  This construction is done by recursion on levels.  First, since $t\in \Delta(\mathcal{S})$ we have $t=|x^\varnothing-x^1|$ with $x^\varnothing,x^1$ special.  In particular, since $x^1$ is a special point, we may apply Lemma \ref{star} with $y=x^\varnothing,x=x^1$
to generate a sequence $x^2,x^3,\dots$ of distinct special points satisfying $|x^\varnothing-x^i|=t$ for every $i$ (Figure \ref{l1}).  This gives the construction of levels $L_0$ and $L_1$.
\begin{figure}[h!]
\begin{subfigure}[b]{0.45\linewidth}
\centering
\begin{tikzpicture}
\draw [fill] (2.967,3.657) circle [radius=0.1];
\draw [fill] (2,1) circle [radius=0.1];
\draw [fill] (1.5,1.239) circle [radius=0.1];
\draw [fill] (1,1.625) circle [radius=0.1];
\draw [fill] (.5,2.274) circle [radius=0.1];

\draw [blue] (2,1) arc [radius=2.828, start angle=250, end angle= 180];

\node[below left] at (2,1) {$x^1$};
\node[above right] at (3.067,3.657) {$x^\varnothing$};
\node[below left] at (1.5,1.239) {$x^2$};
\node[below left] at (1,1.625) {$x^3$};
\node[below left] at (.5,2.274) {$x^4$};

\draw (2.967,3.657)--(2,1);
\draw (2.967,3.657)--(1.5,1.239);
\draw (2.967,3.657)--(1,1.625);
\draw (2.967,3.657)--(.5,2.274);
\end{tikzpicture}
\caption{First level}
\label{l1}
\end{subfigure}
\begin{subfigure}[b]{0.45\linewidth}
\centering
\begin{tikzpicture}
\draw [fill] (2.967,3.657) circle [radius=0.1];
\draw [fill] (2,1) circle [radius=0.1];
\draw [fill] (1.5,1.239) circle [radius=0.1];
\draw [fill] (1,1.625) circle [radius=0.1];
\draw [fill] (.5,2.274) circle [radius=0.1];

\draw [fill] (3.5,2.948) circle [radius=0.1];
\draw [fill] (3.75,2.286) circle [radius=0.1];


\draw[red] (2.967,3.657) arc [radius=2.828, start angle=45.951, end angle= 0];

\node[below left] at (2,1) {$x^1$};
\node[above right] at (3.067,3.657) {$x^\varnothing$};
\node[below left] at (1.5,1.239) {$x^2$};
\node[below left] at (1,1.625) {$x^3$};
\node[below left] at (.5,2.274) {$x^4$};
\node[above right] at (3.5,2.948) {$x^{(3,1)}$};
\node[above right] at (3.75,2.286) {$x^{(3,2)}$};

\draw[dashed] (2.967,3.657)--(2,1);
\draw[dashed] (2.967,3.657)--(1.5,1.239);
\draw[dashed] (2.967,3.657)--(1,1.625);
\draw[dashed] (2.967,3.657)--(.5,2.274);

\draw (1,1.625)--(3.5,2.948);
\draw (1,1.625)--(3.75,2.286);
\end{tikzpicture}
\caption{Second level descendants of $x^3$}
\label{l2}
\end{subfigure}
\caption{Constructing a short tree}
\end{figure}
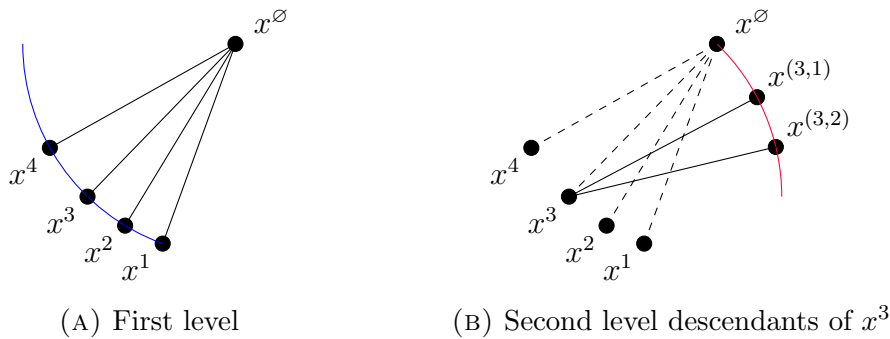  
Next, let $n\geq 1$ and suppose we have constructed distinct special points $x^w$ for every $w\in L_0 \cup\dots\cup L_n$ 
such that for any $w'\in L_{n-1}$ with $w\sim w'$, we have $|x^w-x^{w'}|=t$.  Fix some finite sequence $w=(i_1,\dots,i_n)$, and let $w'=(i_1,\dots,i_{n-1})\in L_{n-1}$.  Since $|x^{w}-x^{w'}|=t$ and $x^{w'}$ is special, we may apply Lemma \ref{star} with $x=x^{w'},y=x^w$ to obtain a sequence of distinct special points $x^{(w,1)},x^{(w,2)},\dots$ such that $|x^{w,i}-x^w|=t$ for every $i$.  Moreover, since Lemma \ref{star} gives uncountably many choices but we have only constucted countably many points so far, we may choose our new points to be distinct from all points at previous levels and not just from each other.  This allows us to construct $L_{n+1}(w)$, the points in level $n+1$ which are children of $w$.  Doing this for every $w\in L_n$ and ensuring no points are chosen more than once, we construct the entire level $L_{n+1}$.
\end{proof}

\bibliography{bibdatabase}
\bibliographystyle{abbrv}

\end{document}